\documentclass[11pt]{amsart}
\usepackage{doc}

\usepackage[top=3cm, bottom=3cm, left=3cm, right=3cm]{geometry}
\usepackage{hyperref} 
\usepackage{tikz-cd}

\hypersetup{
  pdftitle={Univalence of a certain quartic},
  pdfauthor={Jimmy Dillies},
  pdfsubject={},
  pdfkeywords={},
  colorlinks=true,
  breaklinks=true,
  bookmarksopen=true,
  bookmarksnumbered=true,
  pdfpagemode=UseOutlines,
  plainpages=false,
  linkcolor=blue,
  anchorcolor=red,
  citecolor=blue}

\usepackage{amssymb,amsmath}
\usepackage[applemac]{inputenc}

\usepackage{braket}

\newtheorem{prop}{Proposition}

\newcommand{\IC}{\mathbb{C}}
\newcommand{\IN}{\mathbb{N}}

\title{Univalence of a certain quartic function} % Article title

\author{Jimmy Dillies}
\address{Department of Mathematics, Georgia Southern University, 
Statesboro, GA 30460, U.S.A.}
\email{jdillies@georgiasouthern.edu}
\urladdr{http://jimmy.klacto.net/}
\thanks{The author would like to thank Pack 935 for its spartan yet warm hospitality while this note was written.}
\date{}


\begin{document}

\maketitle 

%----------------------------------------------------------------------------------------
%	ABSTRACT
%----------------------------------------------------------------------------------------

\begin{abstract}
\noindent We give a short proof that the quartic polynomial $f(z)=\frac 1 6 z^{4} + \frac 2 3 z^{3} + \frac 7 6 z^{2} + z$ is univalent, i.e. injective in the open unit disc, $D=\Set{ z \in \IC |  \lVert z \rVert <1 }$.
\end{abstract}

%----------------------------------------------------------------------------------------
%	BODY
%----------------------------------------------------------------------------------------

\section{Introduction}

In~\cite{DSSS18}, Dmitrishin \emph{et al.} study the relation between the stability of equilibrium in discrete dynamical systems and a problem of optimal covering of the interval $(-\mu, \mu)$ by the inverse of the polynomial image of the unit disc. 
In their work, the univalence of a family of polynomials indexed by the natural numbers, $f_N(z), N\in \IN$, has been mooted as a crucial ingredient to examine the stability of systems. 
In this note, we show that the polynomial $f_4(z)$, the first case which they do not treat, is univalent. 
This same polynomial appears in two unpublished preprints of Gluchoff and Hartmann~\cite{GluHart02,GluHart06} where they appear as extremal cases of \emph{starlike} polynomials and use notions of stability from physics as was suggested earlier by Alexander~\cite{Ale15} (see~\cite{GluHart00} for a detailed overview).
Our argument is more elementary and follows from the decomposition of $f_4$ into two quadratic functions.
Also, we show how univalence applies more generally to quartic polynomials which are composite of $\tau_{a\geq 1}(z)=(z+a)^2$.

\section{Proof}

\subsection{Decomposition.}

Consider the quartic polynomial $f(z)=f_4(z)=\frac 1 6 z^{4} + \frac 2 3 z^{3} + \frac 7 6 z^{2} + z$.
From the location of its roots it is easy to deduce that the function can be decomposed as $f(z)=q\left( (z+1)^{2} \right)$ where 
$$q(z)=\frac 1 6 \left( z^{2} +z -2 \right).$$

\begin{figure}[h]
\centering
\includegraphics[width=0.15\textwidth]{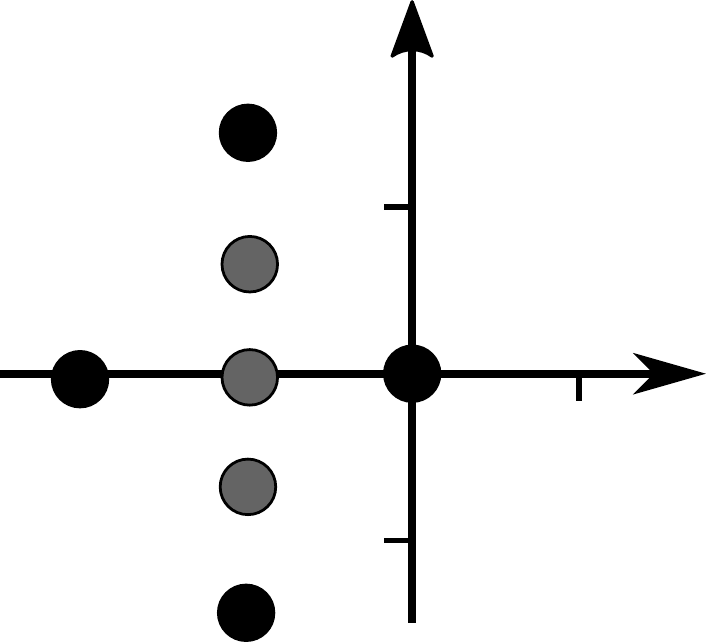}
\caption{Roots of $f$ (black) at $0,-2,-1\pm \sqrt 2 i $ and $f'$ (grey) at $-1,-1 \pm \frac {\sqrt 2} 2 i$. }
\end{figure}

\subsection{Injectivity}
Given the above decomposition, we can see that if $f(z)=f(w)$ then one of the three criteria below must be met :
\begin{enumerate}
\item $z=w$
\item $(z+1)^{2}=(w+1)^{2}$
\item $(z+1)^{2}+(w+1)^{2}=-1$.
\end{enumerate}
The first two criteria speak for themselves.
The third criterion describes the scenario where $(z+1)^{2}$ and $(w+1)^{2}$ are symmetric with respect to the apex of $q(z)$, located at $- \frac 1 2$.

\subsection{Main result.}
\begin{prop}The function $f$ is univalent.
\end{prop}

\begin{proof}
Let us consider the above three cases:
\begin{enumerate}
\item No comment.
\item From the factorization of $(z+1)^{2}-(w+1)^{2}$ we see that unless they are equal, $z$ and $w$ cannot lie simultaneously in $D$ as their horizontal separation is $2$.
\item By contradiction, one of the two terms in the sum ought to have a real part less than $-\frac 1 2$. 
Without loss of generality let this be $(z+1)^{2}$.
Now, the regions defined by $\Re\left( (z+1)^{2} \right) \leq - \frac 1 2 \equiv (x+1)^{2}-y^{2} \leq - \frac 1 2$, where $z=x+iy$ is the Cartesian decomposition, and $D$ are bordered respectively by a hyperbola and a disc which are tangent at the points $- \frac 1 2 \pm \frac{\sqrt{3}}{2} $ and do not overlap.
Henceforth, $z$ and $w$ do not both lie in the open unit disc.
\end{enumerate}
\end{proof}

\subsection{Remark.} The above result can be generalized \emph{mutatis mutandis} to real quartics of the form
 $$q((z+a)^{2})$$ 
 where $|a| \geq 1$. 
 When $|a| > \sqrt{2}$ there is no constraint on $q$; for values $1 < |a| < \sqrt{2}$ one needs to pay some extra care to the location of the roots of $q$, namely univalence requires that $C>\frac{2-a^{2}}{4}$ where $-C$ is the sum of the roots.

%----------------------------------------------------------------------------------------
%	REFERENCES
%----------------------------------------------------------------------------------------

\end{document}